\documentclass[11pt]{amsart}

\usepackage{amsmath,amssymb,amscd,amsthm,graphicx,enumerate}

\usepackage{xcolor}
\usepackage{hyperref}
\hypersetup{breaklinks=true}

\numberwithin{equation}{section}
\theoremstyle{plain}

\makeatletter
\newtheorem*{rep@theorem}{\rep@title}
\newcommand{\newreptheorem}[2]{%
\newenvironment{rep#1}[1]{%
 \def\rep@title{#2 \ref{##1}}%
 \begin{rep@theorem}}%
 {\end{rep@theorem}}}
\makeatother

\newtheorem{theorem}[equation]{Theorem}
\newreptheorem{theorem}{Theorem}

\newtheorem{claim}[equation]{Claim}

\theoremstyle{remark}
\newtheorem{remark}[equation]{Remark}

\theoremstyle{definition}

\newcommand{\eps}{\varepsilon}

\begin{document}

\title[Revisiting ancient noncollapsed flows in $\mathbb{R}^3$]{Revisiting ancient noncollapsed flows in $\mathbb{R}^3$}
\author{Kyeongsu Choi}
\author{Robert Haslhofer}

\begin{abstract}
In this short paper, we give a new proof of the classification theorem for noncompact ancient noncollapsed flows in $\mathbb{R}^3$ originally due to Brendle-Choi (Inventiones 2019). Our new proof directly establishes selfsimilarity by combining the fine neck theorem from our joint work with Hershkovits and the rigidity case of Hamilton's Harnack inequality.
\end{abstract}

\maketitle

\section{Introduction}

To analyze singularities of the mean curvature flow it is crucial to understand ancient noncollapsed flows \cite{Andrews_noncollapsing,ShengWang,HaslhoferKleiner_meanconvex}, i.e. ancient mean-convex flows $M_t$, such that at any $p\in M_t$ the inscribed and exterior radius is at least $\alpha/H(p,t)$ for some uniform $\alpha>0$.   In particular, in $\mathbb{R}^3$ it is known thanks to \cite{CHH,CCS,BK_mult_one} that all blowup limits near all generic singularities are ancient noncollapsed flows. In the noncompact case,\footnote{By an important result of Angenent-Daskalopoulos-Sesum \cite{ADS1,ADS2} the only compact solutions are the round shrinking spheres and the ancient ovals from \cite{White_nature}.} these flows have been classified by Brendle and the first author:

\begin{theorem}[Brendle-Choi \cite{BC}]\label{thm_BC}
Any noncompact ancient noncollapsed flow in $\mathbb{R}^3$ is, up to scaling and rigid motion, either the static plane, the round shrinking cylinder, or the rotationally symmetric bowl.
\end{theorem}

The theorem and the techniques introduced in its proof had far reaching consequences, including in particular the resolution of Ilmanen's mean-convex neighborhood conjecture \cite{CHH}, Huisken's genericity conjecture \cite{CM_generic,CCMS,CCS,CCMS_rev}, and White's two-sphere conjecture \cite{Brendle_sphere,HershkovitsWhite,CHH,BK_mult_one}. Moreover, it also led to a corresponding classification of $\kappa$-solutions in 3d Ricci flow \cite{Brendle_3dRicci,ABDS,BDS}, which confirmed a conjecture of Perelman \cite{Perelman1,Perelman2}, and in turn played a key role in the uniqueness of 3d Ricci flow through singularities \cite{BK_uniqueness}, and its topological and geometric applications \cite{BK_top1,BK_top2,BK_contr}.\\

Here, we give a new proof of Theorem \ref{thm_BC}, inspired by our recent classification of ancient noncollapsed flows in $\mathbb{R}^4$ \cite{CH_ancient_r4} (see also \cite{CHH_wing,CHH_translator,DH_shape,DH_no_rotation,CDDHS}). To describe our approach, let $\mathcal{M}=\{ M_t \}$ be any noncompact ancient noncollapsed flow in $\mathbb{R}^3$. We can assume that it does not split off a line, since otherwise it would be a static plane or a round shrinking cylinder. Moreover, by a rigid motion we can arrange that the axis of its cylindrical tangent flow at $-\infty$ is along the $x_1$-coordinate. Given any space-time point $X=(x,t)\in\mathcal{M}$, we then consider the centered renormalized flow $\bar{M}_\tau^X=e^{\tau/2}(M_{t-e^{-\tau}}-x)$. For $\tau$ sufficiently negative, depending only on the cylindrical scale $Z(X)$, the surface $\bar{M}_\tau^X$ can locally be expressed as graph of a function $u^X(y,\vartheta,\tau)$ over the normalized round cylinder $\mathbb{R}\times S^1(\sqrt{2})$. Considering the slope $u_y^X=\partial u^X/\partial y$, by differentiating the fine neck theorem from our joint work with Hershkovits \cite{CHH}, we first observe:

\begin{theorem}[differential neck theorem]\label{thm_DN}
There exists a constant $a=a(\mathcal{M})\neq 0$, such that for every $X\in\mathcal{M}$ for all $\tau\leq \mathcal{T}(Z(X))$ we have
\begin{equation}\label{diff_neck_exp}
\sup_{|y|\leq 10} \left|  u_y^X(y,\vartheta,\tau)-a e^{\tau/2} \right| \leq C e^{(\frac{1}{2}+\delta)\tau}.
\end{equation}
Here, $\delta>0$ and $C=C(\mathcal{M})<\infty$ are constants, and $\mathcal{T}:\mathbb{R}_+\to \mathbb{R}_{-}$ (possibly depending on $\mathcal{M}$) is a decreasing function.
\end{theorem}

The constant $a=a(\mathcal{M})$ captures how size of the circular fibres changes as the observer moves along the $x_1$-axis. For example, for the rotationally symmetric bowl translating with unit speed in positive $x_1$-direction one has $a=1/\sqrt{2}$. Combining Theorem \ref{thm_DN} and the rigidity case of Hamilton's Harnack inequality from \cite{Hamilton_Harnack}, we then directly establish selfsimilarity:

\begin{theorem}[selfsimilarity]\label{thm_selfsim}
Any noncompact ancient noncollapsed flow in $\mathbb{R}^3$ that does not split off a line is selfsimilarly translating (and hence by \cite{Haslhofer_bowl}, up to scaling and rigid motion, is the rotationally symmetric  bowl).
\end{theorem}

To describe this, by scaling we can assume that $a(\mathcal{M})=1/\sqrt{2}$. Consider the cap at time $t$, namely the unique $p_t \in M_t$ where $x_1$ is minimal. Denoting by $\mathfrak{t}(h)$ the time when the cap arrives at level $x_1=h$, this gives rise to a unique space-time point $P_h=(p_{ \mathfrak{t}(h)},\mathfrak{t}(h) )$. Via blowup arguments we then show that the cylindrical scale $Z(P_h)$ and the mean curvature $H(P_h)$ are uniformly bounded for $h\to \pm \infty$, as otherwise we could obtain a contradiction with Theorem \ref{thm_DN}. Since for $h \to \pm \infty$ we have convergence to translators with slope parameter $a=1/\sqrt{2}$, hence with speed $1$, the rigidity case of Hamilton's Harnack inequality then allows us to conclude.

\begin{remark}[higher dimensions]
Using the fine neck theorem from our joint work with Hershkovits and White \cite{CHHW} in lieu of the one from \cite{CHH}, our argument also yields a new proof of the classification theorem for uniformly two-convex solutions in $\mathbb{R}^{n+1}$ from \cite{BC2}.
\end{remark}

\noindent\textbf{Acknowledgments.}
KC has been supported by the KIAS Individual Grant MG078902, an Asian Young Scientist Fellowship, and the National Research Foundation (NRF) grants RS-2023-00219980 and RS-2024-00345403 funded by the Korea government (MSIT). RH has been supported by the NSERC Discovery grants RGPIN-2016-04331 and RGPIN-2023-04419.

\bigskip

\section{Background and differential neck theorem}

Let $\mathcal{M}=\{ M_t \}$ be any noncompact ancient noncollapsed flow in $\mathbb{R}^3$ that does not split off a line (equivalently, by \cite{HaslhoferKleiner_meanconvex} this means that $M_t$ is strictly convex). By general theory \cite{CM_uniqueness,HaslhoferKleiner_meanconvex}, the tangent flow at $-\infty$ is a round shrinking cylinder, namely in suitable coordinates we have
\begin{equation}\label{tangent_flow}
\lim_{\lambda \rightarrow 0} \lambda M_{\lambda^{-2}t}=\mathbb{R}\times S^{1}(\sqrt{-2t}).
\end{equation}
In other words, given any space-time point $X=(x,t)\in\mathcal{M}$, the centered renormalized flow,
\begin{equation}
\bar{M}_\tau^X=e^{\tau/2}(M_{t-e^{-\tau}}-x),
\end{equation}
converges for $\tau\to -\infty$ to normalized round cylinder $\Gamma=\mathbb{R}\times S^1(\sqrt{2})$. We can thus express $\bar{M}^X_\tau$ as graph of a function $u^X(\cdot,\tau)$ over $\Gamma\cap B_{\rho(\tau)}$, namely
\begin{equation}\label{graph_cyl}
\left\{ q+u^X(q,\tau)\nu(q)\, : \, q\in \Gamma\cap B_{\rho(\tau)}   \right\}\subset \bar{M}_\tau^X,
\end{equation}
where $\nu$ denotes the outwards unit normal of $\Gamma$, and $\lim_{\tau\to-\infty}\rho(\tau)= \infty$. More quantitatively, \eqref{graph_cyl} holds for all $\tau\leq - 2\log Z(X)$. Here, similarly as in \cite[Section 3.2]{CHH}, fixing a small constant $\eps_0>0$, the cylindrical scale $Z(X)$ is defined as the smallest $r<\infty$, such that the parabolically dilated flow $\mathcal{D}_{1/r}(\mathcal M -X)$ is $\varepsilon_0$-close in $C^{\lfloor1/\varepsilon_0 \rfloor}$ in $B(0,1/\varepsilon_0)\times [-1,-2]$ to the evolution of a round shrinking cylinder with radius $r(t)=\sqrt{-2t}$.

Since $\bar{M}^X_\tau$ moves by renormalized mean curvature flow, the evolution of $u^X=u^X(y,\vartheta,\tau)$ is governed by the Ornstein-Uhlenbeck operator
\begin{equation}
\mathcal{L}=\partial^2_y -\tfrac12 y \partial_y +\tfrac12 \partial^2_\vartheta +1,
\end{equation}
see e.g. \cite{Ecker_logsob}. This is a self-adjoint operator on the Hilbert space $\mathcal{H}=L^2(\Gamma,e^{-q^2/4}dq)$, and the only unstable eigenfunctions are the constant function $1$ with eigenvalue $1$ and the functions $y$, $\cos\vartheta$, $\sin\vartheta$ with eigenvalue $1/2$.\\

Based on the above spectral properties, it has been shown in \cite[Section 2]{BC} that for $\tau\to -\infty$ one has
\begin{equation}\label{eq_fine_neck_decay}
\| u^X(\cdot,\tau) \|_{\mathcal{H}}=O(e^{\tau/2}).
\end{equation}
The estimate \eqref{eq_fine_neck_decay} has been derived again (in a more general setting, which is not needed for our purpose) in \cite[Proposition 4.11]{CHH}. Pushing this a bit further, it has then been shown in \cite[Theorem 4.15]{CHH} that in fact 
\begin{equation}\label{eq_fine_neck_exp}
\sup_{|y|\leq 100} \left| e^{-\tau/2} u^X(y,\vartheta,\tau)- a  y - b^X \cos\vartheta - c^X \sin\vartheta \right| \leq Ce^{\delta \tau}
\end{equation}
for all $\tau\leq \mathcal{T}(Z(X))$, where $\delta>0$ and $C=C(\mathcal{M})<\infty$ are constants, and $\mathcal{T}:\mathbb{R}_+\to \mathbb{R}_{-}$  is a decreasing function, which may depending on $\mathcal{M}$. Here, the slope parameter $a=a(\mathcal{M})\neq 0$ is independent of the center $X$, while the expansion parameters $b^X$ and $c^X$ may depend on the center $X$. Note that the lemmas in \cite[Section 4.1]{CHH} were already derived in \cite[Section 2]{BC} for noncollapsed flows. Thus, the estimate \eqref{eq_fine_neck_exp} can be proven by concatenating \cite[Section 2]{BC} and \cite[Section 4.2.2 \& 4.2.3]{CHH}.

Differentiating \eqref{eq_fine_neck_exp}, which is justified thanks to standard Schauder estimates (see e.g. \cite{Lieberman}), we immediately get the differential neck expansion
\begin{equation}\label{eq_diff_neck_exp}
\sup_{|y|\leq 10} \left| e^{-\tau/2} u^X_y(y,\vartheta,\tau)-  a  \right| \leq Ce^{\delta \tau}
\end{equation}
for all $\tau\leq \mathcal{T}(Z(X))$. Note that the expansion \eqref{eq_diff_neck_exp} is equivalent to the expansion we stated in Theorem \ref{thm_DN} (differential neck theorem).

\begin{remark}[intuitive explanation]
To better understand the differential neck theorem, note that in light of the commutator identity $[\mathcal{L},\partial_y]=\tfrac12 \partial_y$ the evolution of $u^X_y$ is governed by the operator $\mathcal{L}'=\mathcal{L}-\tfrac12$. The only unstable eigenfunction of $\mathcal{L}'$ is the constant function $1$ with eigenvalue $1/2$, which explains that $u_y^X \sim a^X e^{\tau/2}$. Also, since for $\tau\to -\infty$ any two necks overlap, we see that the slope $a^X$ is in fact independent of the center $X$.
\end{remark}

\bigskip

\section{Proof of the classification theorem}

Let us first recall a few general facts from the theory of noncollapsed flows \cite{HaslhoferKleiner_meanconvex}. Any ancient noncollapsed flow is strictly convex unless it splits off a line. Moreover, for any ancient noncollapsed curve shortening flow the tangent flow at $-\infty$ is a static line or a round shrinking circle. Hence, any ancient noncollapsed flow in $\mathbb{R}^3$ that splits off a line is either a static plane or a round shrinking cylinder. Furthermore, we also recall that it has already been shown ten years ago by the second author \cite{Haslhofer_bowl}, inspired by Brendle's uniqueness result for steady 3d Ricci solitons \cite{Brendle_steady}, that any selfsimilarly translating noncollapsed flow in $\mathbb{R}^3$ that does not split off a line is, up to scaling and rigid motion, the rotationally symmetric bowl.
Thus, to give a new proof of Theorem \ref{thm_BC}, it suffices to prove Theorem \ref{thm_selfsim}.

\begin{proof}[{Proof of Theorem \ref{thm_selfsim}}]
Let $\mathcal{M}=\{ M_t \}$ be any noncompact ancient noncollapsed flow in $\mathbb{R}^3$ that does not split off a line. By a rigid motion and scaling we can assume that its tangent flow at $-\infty$ is given by \eqref{tangent_flow} and that the slope parameter from Theorem \ref{thm_DN} is given by $a(\mathcal{M})=1/\sqrt{2}$. The expansion \eqref{diff_neck_exp} and convexity imply that $\inf_{p\in M_t} x_1(p)> -\infty$. By strict convexity this infimum is attained at a unique point $p_t \in M_t$. Now, for any $h\in\mathbb{R}$ we denote by $\mathfrak{t}(h)$ the unique $t$ such that $x_1(p_t)=h$. This gives rise to the space-time point $P_h=(p_{\mathfrak{t}(h)},\mathfrak{t}(h))$, which we call the cap at level $h$. To proceed, we need a uniform bound for the cylindrical scale.

\begin{claim}[{cylindrical scale, c.f. \cite[Proposition 5.8]{CHH} and \cite[Proposition 6.2]{CHHW}}]\label{claim_cyl}
We have $\limsup_{h\to \pm \infty} Z(P_h) < \infty$.
\end{claim}

\begin{proof}[{Proof of Claim \ref{claim_cyl}}]
Suppose towards a contradiction that $h_i$ is a sequence converging to infinity or minus infinity, such that $Z(P_{h_i})\to \infty$. Let $\mathcal{M}^i=\mathcal{D}_{1/Z(P_{h_i})}(\mathcal M -P_{h_i})$ be the sequence of flows that is obtained from $\mathcal{M}$ by shifting the cap at level $h_i$ to the space-time origin and parabolically rescaling by $Z(P_{h_i})^{-1}$. Then, by \cite[Theorem 1.14]{HaslhoferKleiner_meanconvex} along a subsequence we can pass to a limit $\mathcal{M}^\infty$, which is a noncompact ancient noncollapsed flow, whose tangent flow at $-\infty$ is a round shrinking cylinder. 

If $\mathcal{M}^\infty$ was a round shrinking cylinder, then if it became extincting at time $0$ we would obtain a contradiction with the definition of the cylindrical scale, and if it became extinct at some later time we would obtain a contradiction with the fact that $M^\infty_t$ is contained in a halfspace for $t>0$. 

Hence, we can apply Theorem \ref{thm_DN}, which gives us a constant $a^\infty\neq 0$ associated to $\mathcal{M}^\infty$, such that \eqref{diff_neck_exp} holds.
However, this contradicts the fact that $a(\mathcal{M}^i)=Z(P_{h_i})^{-1}a(\mathcal{M})\to 0$, and thus proves the claim.
\end{proof}

Continuing the proof of the theorem, recall that by Hamilton's Harnack inequality \cite{Hamilton_Harnack} for every vector field $V$ we have
\begin{equation}\label{Ham_Harnack}
\partial_t H + 2\nabla_V H + A(V,V)\geq 0.
\end{equation}
In particular, inserting $V=0$ one gets $\partial_t H\geq 0$, so $h\mapsto H(P_h)$ is monotone.

Now, thanks to Claim \ref{claim_cyl} and the rigidity case of Hamilton's Harnack inequality \cite{Hamilton_Harnack} for some $h_i\to -\infty$ the sequence $\mathcal{M}^{i}=\mathcal{M}-P_{h_i}$ converges to a noncollapsed translator, which satisfies \eqref{diff_neck_exp} with $a=1/\sqrt{2}$, and hence translates with speed $1$. This implies $\lim_{h\to -\infty}H(P_h)= 1$.

Next, if we had $H(P_{h_i})\to \infty$ along some sequence $h_i\to \infty$, then by \cite[Theorem 1.14]{HaslhoferKleiner_meanconvex} along a subsequence the flows $\mathcal{M}^{i}=\mathcal{M}-P_{h_i}$ would converge to an ancient noncollapsed limit flow $\mathcal{M}^\infty$ that becomes extinct at time $0$. Moreover, since $H(P_{h_i})\to \infty$, for any $t<0$ the domain $K_t^\infty$ bounded by $M_t^\infty$ would contain a line in $x_1$-direction. Hence, $\mathcal{M}^\infty$ would be a round shrinking cylinder, contradicting the expansion \eqref{diff_neck_exp}. Consequently, arguing as in the previous paragraph we infer that $\lim_{h\to \infty}H(P_h)= 1$.
  
We have thus shown that $H(P_h)\equiv 1$, and this implies the assertion. Indeed, by inserting $V=-A^{-1}\nabla H$ in \eqref{Ham_Harnack}, in addition to $H(P_h)\equiv 1$, we get that $\nabla H (P_h)\equiv 0$.
Therefore, by the rigidity case of Hamilton's Harnack inequality  \cite{Hamilton_Harnack} our ancient solution $\mathcal{M}$ is selfsimilarly translating.
\end{proof}

\bigskip

\bibliography{ancient}

\bibliographystyle{alpha}

\vspace{5mm}

{\sc School of Mathematics, Korea Institute for Advanced Study, 85 Hoegiro, Dongdaemun-gu, Seoul, 02455, South Korea}\\

{\sc Department of Mathematics, University of Toronto,  40 St George Street, Toronto, ON M5S 2E4, Canada}\\

\emph{E-mail:} choiks@kias.re.kr, roberth@math.toronto.edu

\end{document}